\documentclass{amsart}
\usepackage[english]{babel}
\usepackage[usenames,dvipsnames]{pstricks}
\usepackage{pst-3dplot,pst-grad}
\usepackage{epsfig}
\usepackage{amssymb}
\usepackage[all]{xy}

\newtheorem{theorem}{Theorem}
\newtheorem{prop}[theorem]{Proposition}
\newtheorem{coro}[theorem]{Corollary}
\newtheorem{remark}[theorem]{Remark}

\newtheorem{definition}[theorem]{Definition}

\newtheorem{example}[theorem]{Example}

\newcommand{\TC}{\mathord{\mathrm{TC}}}
\newcommand{\zcl}{\mathord{\mathrm{zcl}}}
\newcommand{\wTC}{\mathord{\mathrm{wTC}}}
\newcommand{\cTC}{\mathord{\mathrm{conil}_{\TC}}}
\newcommand{\cat}{\mathord{\mathrm{cat}}}
\newcommand{\wcat}{\mathord{\mathrm{wcat}}}
\newcommand{\conil}{\mathord{\mathrm{conil}}}
\newcommand{\cl}{\mathord{\mathrm{cl}}}
\newcommand{\inside}{\mathord{\mathrm{int}}}
\newcommand{\pr}{\mathord{\mathrm{pr}}}
\newcommand{\ev}{\mathord{\mathrm{ev}}}

\newcommand{\pp}{\ltimes}
\newcommand{\nil}{\mathord{\mathrm{nil}}}

\newcommand{\iso}{\cong}

\title{Lower bounds for topological complexity}
\author[Aleksandra Franc]{Aleksandra Franc}
\address{Faculty of Computer and Information Science, University of Ljubljana\newline\indent Tr\v{z}a\v{s}ka 25\newline\indent 1000 Ljubljana, Slovenia}
\email{\rm{aleksandra.franc@fri.uni-lj.si}}
\author[Petar Pave\v{s}i\'c]{Petar Pave\v{s}i\'c}
\address{Faculty of Mathematics and Physics, University of Ljubljana\newline\indent Jadranska 21\newline\indent 1000 Ljubljana, Slovenia}
\email{\rm{petar.pavesic@fmf.uni-lj.si}}
\thanks{This work was supported by the Slovenian Research Agency grant P1-0292-0101; the first author was fully supported under contract No. 1000-07-310002.}
\begin{document}
\begin{abstract}
We introduce fibrewise Whitehead and Ganea definitions of monoidal topological complexity. We then define 
several lower bounds which improve on the standard lower bound in terms of nilpotency of the cohomology 
ring. Finally, the relationships between these lower bounds are studied.
\end{abstract}
\maketitle


\section{Introduction}

The notion of topological complexity was first introduced by Farber \cite{Farber:TC}. It is a homotopy invariant 
that for a given space $X$ measures the complexity of the problem of determining a path in $X$ connecting two points 
in a manner that is continuously dependent on the end-points. In order to give a formal definition observe that the 
map that assigns to a path $\alpha\colon I\to X$ its end-points $\alpha(0)$ and $\alpha(1)$ determines the 
\emph{path-fibration} $\ev_{0,1}\colon X^I\rightarrow X\times X$. A continuous choice of paths between given 
end-points corresponds to a continuous section of this fibration.  

\begin{definition}\label{FarberDef}
{\em Topological complexity} $\TC(X)$ of a space $X$ is the least integer $n$ for which there exist an open 
cover $\{U_1, U_2,\ldots, U_n\}$ of $X\times X$ and sections $s_i\colon U_i\rightarrow X^I$ of the fibration 
$\ev_{0,1}\colon X^I\rightarrow X\times X$, $\alpha\mapsto (\alpha(0),\alpha(1))$.
\end{definition}

$$\xymatrix{
 & X^I\ar[d]^-{\ev_{0,1}}\\
U_i\ar@{-->}[ur]^-{s_i} \ar[r]_-{\subseteq} & X\times X
}$$

If we additionally require that $s_i(x,x)=c_x$ (i.e. the sections map into constant paths over the diagonal $\Delta(X)$), 
we get the definition of the monoidal topological complexity $\TC^M$ of Iwase and Sakai \cite[Definition 1.3]{IS}.

Topological complexity and its variants are still an intensely studied topic. For a sampler of results one can consult 
Chapter 4 of Farber's book \cite{Farber:Robotics}.

Clearly, $\TC(X)=1$ if and only if $X$ is contractible. In general, the determination of the topological 
complexity is a non-trivial problem even for simple spaces like the spheres. The standard approach is to 
find estimates of the topological complexity in the form of upper and lower bounds that are more easily 
computable. With some luck we then obtain a precise value of $\TC$ or at least a restricted list of possibilities.

The most common upper bounds are based on the Lusternik-Schnirelmann category (LS-category) of the product 
space $X\times X$, and on certain product inequalities. On the other side, the standard lower bound for the 
topological complexity is the \emph{zero-divisors cup length} $\zcl(X)$ introduced by Farber \cite{Farber:TC} 
and defined as follows. One takes the cohomology of $X$ with coefficients in some field $k$ and defines the 
\emph{zero-divisors cup length} $\zcl_k(X)$ to be the biggest $n$ such that there is a non-trivial product 
of $n$ elements in the kernel of the cup-product map 
$$\smile\colon H^*(X;k)\otimes H^*(X;k)\to H^*(X;k).$$

We will show that the zero-divisors cup length is only one (and in fact the coarsest) of the lower bounds for 
the topological complexity that can be defined using the characterization of (monoidal) topological complexity by Iwase 
and Sakai \cite{IS}. Their approach is more geometric - they view the monoidal topological complexity (defined below) as 
a special case of the {\em fibrewise pointed Lusternik-Schnirelmann category} as defined by James and Morris 
\cite{JamesMorris}.

In general, a {\em fibrewise pointed space} over a {\em base} $B$ is a topological space $E$, together 
with a {\em projection} $p\colon E\rightarrow B$ and a {\em section} $s\colon B\rightarrow E$. 
Fibrewise pointed spaces over a base $B$ form a category and the notions of fibrewise pointed maps and 
fibrewise pointed homotopies are defined in an obvious way. We refer the reader to \cite{James} and 
\cite{JamesMorris} for more details on fibrewise constructions.

Iwase and Sakai \cite{IS} considered the product $X\times X$ as a fibrewise pointed space over $X$ by taking 
the projection to the first component and the diagonal section $\Delta\colon X\rightarrow X\times X$ as in 
the diagram
$$\xymatrix{
X\times X \ar@<0.5ex>[d]^-{\pr_1}\\
X \ar@<0.5ex>[u]^-{\Delta}
}$$
We will denote this fibrewise pointed space by $X\pp X$. In this case a fibrewise pointed homotopy is any 
homotopy $H\colon X\times X\times I\rightarrow X\times X$ that fixes the first coordinate and is stationary 
on $\Delta(X)$ (i.e. $H(x,y,t)=(x,h(x,y,t))$ for some homotopy $h\colon X\times X\times I\rightarrow X$ that 
is stationary on $\Delta(X)$). The following result of Iwase and Sakai \cite{IS} gives us alternative characterizations 
of $\TC$ and $\TC^M$:

\begin{theorem}\label{IwaseDef}
The topological complexity $\TC(X)$ of $X$ is equal to the least integer $n$ for which there exists an open cover 
$\{U_1, U_2,\ldots, U_n\}$ of $X\times X$ such that each $U_i$ is compressible to the diagonal via a fibrewise homotopy.

The monoidal topological complexity $\TC^M(X)$ of $X$ is equal to the least integer $n$ for which 
there exists an open cover $\{U_1, U_2,\ldots, U_n\}$ of $X\times X$ such that each $U_i$ contains the 
diagonal $\Delta(X)$ and is compressible to the diagonal via a fibrewise pointed homotopy.
\end{theorem}

In \cite{IS} Iwase and Sakai have conjectured that $\TC(X)=\TC^M(X)$ if $X$ is a locally finite simplicial complex. At 
this time a general proof of this conjecture is not known, but they give a slightly weaker result in 
\cite[Theorem 2]{ISerrata}:

\begin{theorem}
Let $X$ be a locally finite simplicial complex. Assume that $\TC(X)=n$ and $\{U_1,\ldots,U_n\}$ is a $\TC$-categorical 
cover of $X\times X$. In the following two cases we have $\TC^M(X)=\TC(X)$:
\begin{enumerate}
\item There exists $j$, $1\leq j\leq n$, such that $U_j$ does not intersect the diagonal $\Delta(X)$ or
\item there exists $j$, $1\leq j\leq n$, such that $U_j$ contains $\Delta(\pr_1(U_j))$.
\end{enumerate}
\end{theorem}

These additional assumptions may be difficult to verify if $\TC$ is known but no explicit cover is given, or if the 
cover provided does not satisfy the conditions, but they also demonstrate that $\TC^M$ and $\TC$ cannot differ by 
more than one \cite[Theorem 1]{ISerrata}:

\begin{theorem}\label{bounds}
Let $X$ be a locally finite simplicial complex. Then $$\TC(X)\leq\TC^M(X)\leq\TC(X)+1.$$
\end{theorem}

In view of this result all the lower bounds we define for the monoidal topological complexity can also be viewed as 
(possibly slightly weaker) lower bounds for topological complexity.

An open set $U\subseteq X\times X$ is said to be {\em fibrewise pointed categorical} if it contains the diagonal 
$\Delta(X)$ and is compressible onto it by a fibrewise pointed homotopy. We may therefore say that $TC^M(X)\le n$ 
if $X\times X$ can be covered by $n$ fibrewise pointed categorical sets, a formulation that is completely 
analogous to the definition of the LS-category. The principal objective of this paper is to develop further 
this point of view and to extend some standard constructions from the LS-category to the new context. In 
particular, we will give a systematic account of lower bounds for the monoidal topological complexity that closely 
follows analogous bounds for the LS-category.

In the first section we will describe fibrewise pointed spaces that will play a role in the later exposition. 
The notation is chosen so as to stress the relation with the fibrewise pointed space $X\pp X$ used in the 
Iwase-Sakai definition. The common framework that justifies the same notation for seemingly disparate constructions 
is given in the Appendix.

In the second section we will describe the Whitehead- and the Ganea-type characterizations of the monoidal topological 
complexity. The former was already known: it appears in \cite{IS} and can be in fact traced back at least to 
\cite[Section 6]{JamesMorris}. As for the latter, this phenomenon has been observed in the case of topological 
complexity which is equivalent to the Schwarz genus \cite{Schwarz} of the path-fibration, a characterization of the 
topological complexity that appears in the original paper \cite{Farber:TC} by Farber. We use both characterizations 
in order to construct a diagram of fibrewise pointed spaces needed to describe lower bounds for the monoidal 
topological complexity.

In the last section we describe a series of estimates for the monoidal topological complexity and derive principal 
relations between them. Variants of some of these lower bounds for topological complexity have already been 
considered in \cite[Section 8]{IS} and \cite{GV}.

Throughout this paper $1$ is used to denote the identity map when the domain can be inferred from the context, 
$\Delta_n\colon A\rightarrow A^n$ denotes the diagonal map $a\mapsto (a,\ldots,a)$ (we write $\Delta$ instead 
of $\Delta_2$ to unburden the already encumbered notation) and $\pr_1\colon A\times B\rightarrow A$ denotes 
the projection to the first component for any spaces $A$, $B$. The interior and the closure of the set $A$ are denoted 
by $\inside(A)$ and $\cl(A)$, respectively. The symbol $\mathbb{N}$ will denote the set of 
positive integers and $\mathbb{N}_0:=\mathbb{N}\cup\{0\}$.


\section{Pointed constructions and fibrewise pointed spaces}\label{sec2}

Several lower bounds for the LS-category can be derived from a diagram that relates the Whitehead and the Ganea 
characterization (see Section 1.6 and Chapter 2 of \cite{CLOT}). To obtain analogous estimates for the monoidal topological 
complexity, we must first understand the fine structure of the fibrewise pointed spaces involved. In all cases 
we obtain fibrewise pointed spaces that can be seen as continuous families of pointed topological spaces whose 
base-points are parametrized by the points of the base space. 

We are going to describe a number of fibrewise pointed spaces that play a role in the fibrewise Whitehead and 
fibrewise Ganea characterizations of $\TC^M$. These are of course only special instances of general fibrewise pointed 
constructions described in \cite[Chapter 2]{Crabb-James}, but in our case they allow a very explicit description 
that will play a role in our exposition. In spite of the different definitions these spaces share some basic 
features and this motivates the common notation. We will provide further justification for this decision in 
the Appendix, where we will describe a general construction, which under favourable circumstances subsumes 
all the above. For now we assume that $(X\times X,\Delta(X))$ is an NDR pair. This condition ensures that $\TC^M$ 
is a homotopy invariant (cf. \cite[Remark 1.4]{IS}).

\ \\
\textbf{Pointed space:}
Let $X\pp X$ denote the fibrewise pointed space 
$X\stackrel{\Delta}{\longrightarrow}X\times X\stackrel{\pr_1}{\longrightarrow}X$. As a space $X\pp X$ is 
therefore just the product $X\times X$ but the choice of the section means that the fibre over each $x\in X$ 
is the pointed topological space $(X,x)$ rather than $(X,x_0)$ for some fixed $x_0\in X$.

\ \\
\textbf{Product:}
\label{ex product}
The categorical product in the category of fibrewise spaces is given by the pull-back construction over the base 
space. If we compute the fibrewise pointed $n$-fold product of $X\pp X$ we obtain the space 
$$\{(x_1,y_1,\ldots,x_n,y_n)\in X^{2n}\;|\;x_1=\ldots=x_n\},$$
with the projection 
$pr_1\colon (x_1,y_1,\ldots,x_n,y_n)\mapsto x_1$ and the diagonal section $\Delta_{2n}$. Since the odd-numbered 
coordinates in the product coincide we may instead take the space
$$\{(x,y_1,\ldots,y_n)\;|\;x, y_i\in X\}=X\times X^n,$$
which is a fibrewise pointed space with respect to the projection $\pr_1$ and the section 
$(1,\Delta_n)\colon X\rightarrow X\times X^n$. We denote this \emph{fibrewise pointed $n$-fold product} by 
$X\pp\Pi^nX$. Observe that the fibre over $x\in X$ is the pointed product space $(X^n,(x,\ldots,x))$. 

\ \\
\textbf{Fat wedge:}
From the $n$-fold product we may extract the \emph{fibrewise pointed fat wedge} by taking the subspace 
$$\{(x,y_1,\ldots,y_n)\in X\pp\Pi^nX\;|\;\exists j:y_j=x\}$$
and restricting the  projection and the section 
accordingly. This fibrewise pointed space is denoted $X\pp W^nX$. The fibre over $x\in X$ is the usual fat 
wedge
$$W^n(X,x)=\{(x_1,\ldots,x_n)\in X^n\;|\;\exists j:x_j=x\}.$$
Observe that, contrary to the first two examples, 
the fibres of $X\pp W^n X$ are in general not homeomorphic. This shows that $X\pp W^nX$ might not be locally trivial. 
There is an obvious fibrewise pointed inclusion $X\pp W^nX\hookrightarrow X\pp \Pi^nX$ (denoted $1\pp i_n$). 

\ \\
\textbf{Smash product:}
The \emph{$n$-fold fibrewise pointed smash product} $X\pp \wedge^nX$ is obtained by taking the fibrewise quotient of 
$X\pp \Pi^nX$ by its subspace $X\pp W^nX$. The projection $\pr \colon X\pp \wedge^nX\to X$ and the section 
$s\colon X \to X\pp \wedge^nX$ are induced by the projection and the section in the $n$-fold product. The quotient 
map $X\pp\Pi^nX\to X\pp\wedge^nX$ is denoted $1\pp q_n$, and the sequence of fibrewise maps and spaces 
$$\xymatrix{X\pp W^nX  \ar@{^{(}->}[r]^{1\pp i_n} & X\pp \Pi^n X \ar[r]^{1\pp q_n} & X\pp \wedge^n X}$$ 
is a fibrewise cofibration in the sense of \cite{Crabb-James}.

\ \\
\textbf{Path space:} 
Let $\ev_0\colon X^I\to X$ be the projection that to every path in $X$ assigns its starting point, and let 
$c\colon X\to X^I$ be the map that to every $x\in X$ assigns the stationary path at the point $x$. The fibre of 
$\ev_0$ over any $x\in X$ is precisely $P_xX$, the space of paths in $X$ based at the point $x$. It is therefore 
natural to denote by $X\pp PX$ the fibrewise pointed space 
$X\stackrel{c}{\longrightarrow} X^I \stackrel{\ev_0}{\longrightarrow}X$. Observe that the fibration 
$\ev_{0,1}\colon X^I\to X\times X$ determines the fibrewise pointed fibration $X\pp PX\to X\pp X$.

\ \\
\textbf{Ganea space:}
By analogy with the above examples one would expect the fibrewise Ganea space $X\pp G_n(X)$ to be a continuous 
family of Ganea spaces $G_n(X)$ (see \cite[Section 1.6]{CLOT}) with the base-points varying as prescribed by a 
section, a map from the underlying base space $X$. This will indeed be the case.

There are several equivalent definitions of the Ganea spaces. The standard Ganea construction is done 
inductively. Start with the fibration $p\colon E\rightarrow B$ with fibre $F$ and replace it by the quotient fibration
$p\colon E/F\rightarrow B$. A theorem by Ganea then states that the fibre of this new fibration is the join 
$E*\Omega B$ of the total space $E$ and the space of loops $\Omega B$ over the base $B$.

The path fibration $\ev_1\colon PX\to X$ is defined to be the first Ganea fibration $p_1\colon G_1(X)\to X$ for the space $X$ with fibre 
$F_1(X)=\Omega X$, and the $(n+1)$st Ganea fibration $p_{n+1}\colon G_{n+1}(X)\to X$ is obtained recursively by applying the Ganea 
construction to the fibration $p_n$. This consists in replacing the induced map $q_n\colon G_n(X)\cup C(F_n(X))\to X$ with a fibration in the 
usual way (see \cite[Definition 1.59]{CLOT}).

For practical reasons, however, we prefer an alternative description of the Ganea space based on the sum-of-fibrations 
operation introduced by Schwarz \cite{Schwarz}. Given two fibrations $p_1\colon E_1\rightarrow B$ and 
$p_2\colon E_2\rightarrow B$, Schwarz defines their sum $p\colon E\rightarrow B$ with
$$E=\{(e_1,t_1,e_2,t_2)\in E_1\times I\times E_2\times I\;|\;p_1(e_1)=p_2(e_2),\max\{t_1,t_2\}=1\}/\sim$$
and $p(e_1,t_1,e_2,t_2)=p_1(e_1)=p_2(e_2)$, where
$$(e_1,0,e_2,1)\sim (e'_1,0,e_2,1)\quad {\rm and}\quad (e_1,1,e_2,0)\sim (e_1,1,e'_2,0).$$
This is the fibrewise join in the category of fibrewise spaces over B.

To obtain the Ganea spaces using this construction, we start with the path fibration $\ev_1\colon PX\to X$ 
(i.e. we consider $PX$ as fibrewise space over $X$) and define the $n$th Ganea fibration to be the $n$-fold fibrewise 
join $PX*_X\ldots*_X PX\to X$. 

The last formulation has a natural fibrewise extension. Instead of $\ev_1\colon PX\rightarrow X$ we use 
$\ev_{0,1}\colon X^I\rightarrow X\times X$ (i.e. the fibrewise pointed fibration $X\pp PX\to X\pp X$ from the previous 
example). We define
$$X\pp G_n(X) = *^{n}_{X\times X}(\ev_{0,1}\colon X^I\rightarrow X\times X)$$
as the $n$-fold fibrewise join. This is a fibrewise pointed space over $X$ with the projection 
$\pr_1\circ *^n_{X\times X}\ev_{0,1}$ and the section $\widehat{G}_n\colon X\rightarrow G_n(X)$ that assigns 
to each $x\in X$ the point in the total space determined by the stationary path in $x$. 

If for any $x\in X$ we restrict this construction to the fibre over $x$, we get the Ganea space $G_n(X,x)$, obtained 
as the $n$-fold fibrewise join of $\ev_1\colon P(X,x)\to (X,x)$, so the fibres are precisely the previously described 
Ganea spaces, and that justifies the notation $X\pp G_n(X)$. 

A third alternative construction of the Ganea spaces obtains them as the pull-back of the Whitehead diagram. This approach 
will be considered in the next section.


\section{Whitehead- and Ganea-type characterizations of $\TC^M$}\label{secWhite}

In this section we use this fibrewise approach to obtain characterizations of the monoidal topological 
complexity that are completely analogous to the classical Whitehead and Ganea characterizations of the LS-category. 
This was partly done in \cite{IS} but we are able to take a step further and use the fibrewise pointed spaces 
described in the preceding section to obtain the fibrewise pointed version of the standard diagram that relates 
the two characterizations (cf. \cite[Section 1.4]{CLOT}). That diagram will be used in Section \ref{sec:bounds} 
to give a unified treatment and comparison of various estimates of the monoidal topological complexity.

We begin with the Whitehead-type characterization. As usual we need some mild topological assumption to prove the 
equivalence with the original definition. 

\begin{theorem}\label{White}
Let $X\times X$ be a simplicial complex and assume that there exists a fibrewise pointed categorical neighbourhood 
$U$ of the diagonal $\Delta(X)$. Then $\TC^M(X)\leq n$ if and only if the map 
$1\pp\Delta_n\colon X\pp X\rightarrow X\pp\Pi^nX$ can be compressed into $X\pp W^nX$ by a fibrewise pointed homotopy.
$$\xymatrix{
 & & X\pp W^nX\ar[d]^-{1\pp i_n}\\
X\pp X \ar[rr]_-{1\pp \Delta_n}\ar@{-->}[urr]^-{g} & & X\pp\Pi^nX
}$$ 
\end{theorem}

\proof
This theorem is a special case of Propositions 6.1 and 6.2 of \cite{JamesMorris}. We summarize the proof given 
there by James and Morris with appropriate adjustments for this particular setting.

First, let $H_t\colon X\pp X\rightarrow X\pp\Pi^n X$ be the fibrewise pointed deformation of $1\pp\Delta_n$ into a 
map $g\colon X\pp X\rightarrow X\pp W^nX$. Let $U$ be the open fibrewise pointed categorical neighbourhood of 
$\Delta(X)$. Define $U_i=g^{-1}(1\pp\pr_i)^{-1}(U)$ where $1\pp\pr_i\colon X\pp\Pi^nX\rightarrow X\pp X$ is 
$i$th projection on the second factor. Then $$(1\pp\pr_i)H_t\colon X\pp X\rightarrow X\pp X$$ compresses $U_i$ 
to $U$ via a fibrewise pointed homotopy. So, $\{U_1,\ldots,U_n\}$ is a fibrewise pointed open categorical cover of 
$X\pp X$ and $\TC^M(X)\leq n$.

Conversely, assume that $\TC^M(X)\leq n$ and let $\{V_1,\ldots,V_n\}$ be the fibrewise pointed open categorical cover 
of $X\pp X$. Let $H_i\colon V_i\times I\rightarrow X\times X$ be the fibrewise pointed nullhomotopies of the inclusions 
$j_i\colon V_i\rightarrow X\pp X$, $i=1,\ldots,n$. The product $X\pp X$ is normal and $\Delta(X)$ is closed in 
$X\pp X$, so there exist open neighbourhoods $W_i$ of $\Delta(X)$ such that $\cl(W)_i\subset V_i$, as well as 
maps $r_i\colon X\pp X\rightarrow I$ with $r_i|_{\Delta(X)}=1$ and $r_i|_{(X\pp X)\setminus W_i}=0$. A fibrewise 
pointed deformation $d_i\colon (X\pp X)\times I\rightarrow X\pp X$ of the identity is given by
$$d_i(x,y,t)=\left\{\begin{array}{ll}(x,y); & (x,y)\in X\pp X\setminus\cl(W)_i,\\H_i(x,y,t\cdot r_i(x,y)); & (x,y)\in V_i.
\end{array}\right.$$
Let $d\colon (X\pp X)\times I\rightarrow X\pp\Pi^nX$ be the fibrewise pointed homotopy defined by the 
relations $(1\pp\pr_i)d=d_i$ for $i=1,\ldots,n$. Then $d$ compresses $1\pp\Delta_n$ into $X\pp W^nX$.
\endproof

The Ganea-type characterization of the LS-category is based on the homotopy pull-back of the diagram used for the 
Whitehead-type characterization. In fact, if we define
$$\overline G_n(X)=\{\alpha\in(\Pi^nX)^I\;|\;\alpha(0)\in W^nX,\alpha(1)\in\Delta_n(X)\}$$ 
then the diagram  
$$\xymatrix{
\overline G_n(X)  \ar[rr] \ar[d] & & W^nX\ar[d]^-{i_n}\\
X \ar[rr]_-{\Delta_n} & & \Pi^nX
}$$ 
(with the obvious projections from $\overline{G}_n(X)$) is the standard homotopy pull-back. The space 
$\overline G_n(X)$ is known to be homotopy equivalent to the previously defined Ganea space $G_n(X)$ (cf. 
\cite[Theorem 1.63]{CLOT}). It is therefore reasonable to consider the fibrewise homotopy pull-back of the diagram 
used in the Whitehead-type characterization for an alternative characterization of the monoidal topological complexity.

We begin by replacing the inclusion $1\pp i_n\colon X\pp W^nX\rightarrow X\pp\Pi^nX$ by the fibration
$$1\pp\overline{\iota}_n\colon X\pp (W^nX \sqcap (\Pi^nX)^I)\rightarrow X\pp\Pi^n X,$$
$$(1\pp\overline{\iota}_n)(x,y_1,\ldots,y_n,\alpha)=(x,\alpha(1)),$$
where the domain space $E=X\pp (W^nX \sqcap (\Pi^nX)^I) $ is given by
$$E=\{(x,y_1,\ldots,y_n,\alpha)\in X\times W^nX\times(\Pi^nX)^I\;|\;\exists j: y_j=x,\alpha(0)=(y_1,\ldots,y_n)\}.$$
The fibrewise pull-back along $1\pp\Delta_n$ gives us the space
$$\overline E=\{(x',y',x,y_1,\ldots,y_n,\alpha)\in X\times X\times E\;|\;(y,y,\ldots,y)=\alpha(1),x'=x\},$$
that can be also written in a more compact form as
$$\overline E=\{(x,\alpha)\in X\pp(\Pi^nX)^I\;|\;(x,\alpha(0))\in X\pp W^nX,\alpha(1)\in\Delta_n(X)\}.$$
Let $p\colon \overline E\to X$ be the projection to the first component, and let the section $s\colon X\to\overline E$ 
be given by $x\mapsto (x,c_x,\ldots,c_x)$. Since the fibre of $p$ over any $x\in X$ is precisely the space 
$\overline{G}_n(X)$ based at the point $(c_x,\ldots,c_x)=\Delta_n(c_x)$, we denote the fibrewise pointed space 
$X\stackrel{s}{\rightarrow} \overline E\stackrel{p}{\rightarrow} X$ by $X\pp \overline{G}_n(X)$. The fibrewise pointed 
spaces $X\pp G_n(X)$ and $X\pp \overline G_n(X)$ are related by a fibrewise map which is a homotopy equivalence when 
restricted to each fibre. By Theorem 3.6 of \cite{Dold} they are fibre-homotopy equivalent. Thus we may conclude that 
the diagram
\begin{equation}\label{pullbackdgm}\xymatrix{
X\pp G_nX \ar[d]_-{1\pp p_n}\ar[rr] & & X\pp W^nX\ar[d]^-{1\pp i_n}\\
X\pp X\ar[rr]_-{1\pp\Delta_n} & & X\pp\Pi^nX
}
\end{equation}
is a fibrewise pointed homotopy pull-back. Since the liftings of $1\pp \Delta_n$ correspond to sections of $1\pp p_n$ 
we have the following 

\begin{coro}\label{Ganea}
$\TC^M(X)$ is the least integer $n$, such that the map
$$1\pp p_n\colon X\pp G_n(X)\rightarrow X\pp X$$
admits a section.
\end{coro}

\begin{remark}
Observe that we have very nearly obtained Farber's original definition of the topological complexity. Indeed, Farber 
defined $\TC(X)$ to be the Schwarz genus of the path-fibration $\ev_{0,1}\colon X^I\to X\times X$, and Schwarz proved 
in \cite{Schwarz} that the genus of a fibration $p\colon E\to B$ is the minimal $n$ such that the $n$-fold fibrewise 
join $p_n\colon E\ast_B\ldots\ast_B E\to B$ admits a section. The space we obtain starting with $\ev_{0,1}$ coincides 
with the fibrewise Ganea space $X\pp G_n(X)$ we have just defined if we consider them both as fibrewise spaces over 
$X$. The difference is, our Whitehead and Ganea definitions are based on fibrewise pointed spaces and we need all 
homotopies to preserve the sections, resulting in the more restrictive $\TC^M$ rather than $\TC$.
\end{remark}

These fibrewise Ganea- and Whitehead-type characterizations of monoidal topological complexity 
(Theorem \ref{White} and Corollary \ref{Ganea}) are closely related to the Ganea and Whitehead 
characterizations of LS-category. Indeed, if we choose any $x\in X$ and restrict all spaces and maps in the diagram above 
to the fibre over $x$, we retrieve the homotopy pull-back diagram for the LS-category. This leads to the following 
diagram of fibrewise pointed spaces over $X$
\begin{equation}\label{kocka}\xymatrixcolsep{1cm}\xymatrixrowsep{1.2cm}
\xymatrix@!0{
& & G_n(X)\ar[rrrr]\ar'[d][dd]\ar[dll] & & & & W^nX\ar[dd]\ar[dll]\\
X\ar[rrrr]\ar[dd] & & & & \Pi^nX\ar[dd]\\
& & X\pp G_n(X)\ar'[rr][rrrr]\ar'[d][dd]\ar[dll] & & & & X\pp W^nX\ar[dd]\ar[dll]\\
X\pp X\ar[rrrr]\ar[dd] & & & & X\pp\Pi^nX\ar[dd]\\
& & X\ar@{=}'[rr][rrrr] & & & & X\\
X\ar@{=}[rrrr]\ar@{=}[urr] & & & & X\ar@{=}[urr]
}
\end{equation}
that allows us to easily compare any invariants that are connected to the existence of sections in this diagram. 
For example, we see that $\cat(X)\leq\TC^M(X)$ since $p_n$ obviously admits a section if $1\pp p_n$ admits a section.


\section{Lower bounds for topological complexity}\label{sec:bounds}

We have already mentioned the most widely used lower bound for the topological complexity, namely the zero-divisors 
cup length. Although the original definition used cohomology with field coefficients this restriction is not necessary. 
Indeed, let $R$ be any ring and let $\nil_R(X)$ be the least $n$ such that all cup-products of length $n$ in the 
ring $H^*(X\times X,\Delta(X);R)$ are trivial. We obtain the following estimate by following the reasoning in 
\cite[\S 3]{JamesMorris}. Note that James and Morris have stated their result for fibrewise pointed LS-category, so 
it clearly holds for $\TC^M$, but it turns out their idea also works for what Iwase and Sakai call fibrewise unpointed 
LS-category, or, in particular, for $\TC$:

\begin{prop}
Let $R$ be a ring and let $X$ be a locally finite simplicial complex. Then $\TC(X)\geq\nil_R(X)$.
\end{prop}

\proof
Note that $\TC(X)$ is the fibrewise (unpointed) LS-category of the fibrewise pointed space 
$X\stackrel{\Delta}{\longrightarrow}X\pp X\stackrel{\pr_1}{\longrightarrow}X$. If the subset $U$ is fibrewise categorical 
in $X\pp X$, then the induced map
$$H^*(X\pp X,\Delta(X))\longrightarrow H^*(U,\Delta(X))$$
is trivial. This is true because $U$ can be compressed into $\Delta(X)$. Note that it does not matter if this homotopy 
is fibrewise pointed, just fibrewise or neither of those. It then follows from the cohomology exact sequence
$$\ldots\rightarrow H^*(X\pp X,U)\rightarrow H^*(X\pp X,\Delta(X))\stackrel{0}{\rightarrow} H^*(U,\Delta(X)) \rightarrow H^*(X\pp X,U) \rightarrow\ldots$$
of the triple $(X\pp X,U,\Delta(X))$ that the map
$$H^*(X\pp X,U)\longrightarrow H^*(X\pp X,\Delta(X))$$
is surjective.
If $\{U_1,\ldots,U_n\}$ is a fibrewise pointed categorical open covering of $X\pp X$, then for every 
$\alpha_i\in H^*(X\pp X,\Delta(X))$ there exists a $\beta_i\in H^*(X\pp X,U_i)$, $i=1,\ldots,n$. The product
$$\beta_1\smile\ldots\smile\beta_n\in H^*(X\pp X,U_1\cup\ldots\cup U_n)$$
is zero since the group itself is trivial. So, the product
$$\alpha_1\smile\ldots\smile\alpha_n\in H^*(X\pp X,\Delta(X))$$
is zero. This shows that $\TC(X)\geq\nil_R(X)$.
\endproof

\begin{remark}
Under the same assumptions we also have $\TC^M(X)>\nil_R(X)$. We can use the same argument as above or simply apply 
Theorem \ref{bounds}.
\end{remark}

\begin{remark}
If we specialize to coefficients in a field $k$ then $\nil_k(X)=\zcl_k(X)+1$. More precisely, if $k$ is a field and 
$H^*(X;k)$ is of finite type, then the external cross product
$$\times\colon H^*(X;k)\otimes_k H^*(X;k)\rightarrow H^*(X\times X;k)$$
is an isomorphism \cite[\S VII.7]{Dold:book}. Also, $x_1\smile x_2=\Delta^*(x_1\times x_2)$ \cite[\S VII.8]{Dold:book}. 
Finally, the cohomology exact sequence 
of the pair $(X\times X,\Delta(X))$ splits into short exact sequences of the form 
$$0\rightarrow H^*(X\times X,\Delta(X);k) \rightarrow H^*(X\times X;k) \stackrel{i^*}{\rightarrow} H^*(\Delta(X);k) \rightarrow 0.$$
From this we obtain the following commutative diagram:
$$\xymatrix{
0\ar[r] &  H^*(X\times X, \Delta(X);k)\ar[r] & H^*(X\times X;k)\ar[r]^-{i^*}\ar@{<-}[d]^-{\times}_-{\cong}\ar[dr]^-{\Delta^*} & H^*(\Delta(X);k)\ar@{=}[d]\\
 & & H^*(X;k)\otimes H^*(X;k)\ar[r]_-{\smile}&  H^*(X;k)
}$$
Under this identification we have $\ker\smile = \ker i^* = H^*(X\times X,\Delta(X);k)$.
\end{remark}

In the theory of LS-category there are several invariants that are better estimates of $\cat(X)$ than the cup-length. 
A systematic exposition of lower bounds for the LS-category in \cite[Chapter 2]{CLOT} is based on the diagram 
relating Whitehead and Ganea characterizations of the LS-category. Based on the results from the previous 
section we are now able to generalize this approach and obtain estimates of the monoidal topological complexity that are 
better than the zero-divisors cup length. We first complete the pull-back diagram (\ref{pullbackdgm}) by adding 
the cofibres of the vertical maps. We have already considered the cofibre of the inclusion 
$1\pp i_n\colon X\pp W^nX\rightarrow X\pp\Pi^nX$, it is the fibrewise pointed space $X\pp \wedge^nX$. We may likewise 
construct the fibrewise pointed space $X\pp G_{[n]}(X)$ as the cofibre of the map 
$1\pp p_n\colon X\pp G_n(X)\rightarrow X\pp X$. The quotient map is denoted 
$1\pp q'_n\colon X\pp X\rightarrow X\pp G_{[n]}(X)$. We obtain the following diagram of fibrewise pointed spaces 
over $X$:
$$\xymatrix{
X\pp G_nX \ar[d]_-{1\pp p_n}\ar[rr]^{1\pp\widehat{\Delta}_n}& & X\pp W^nX\ar[d]^-{1\pp i_n}\\
X\pp X \ar[d]_-{1\pp q'_n}\ar[rr]^-{1\pp\Delta_n}& & X\pp\Pi^nX\ar[d]^-{1\pp q_n}\\
X\pp G_{[n]}X \ar[rr]_{1\pp\tilde{\Delta}_n}& & X\pp\wedge^nX
}$$
Based on the interrelations in the diagram we define several lower bounds for the monoidal topological complexity. All of 
them have analogues among lower bounds for the LS-category and that motivates both the notation and their names.

\begin{definition}
\label{5def}
\begin{enumerate}
\item Let $R$ be a ring. The \emph{$\TC$-Toomer invariant of $X$ with coefficients in $R$}, denoted $e^{\TC}_R(X)$, 
is the least integer $n$ such that the induced map 
$$H_*(1\pp p_n)\colon H_*(X\pp G_n(X),\widehat{G}_n(X);R)\rightarrow H_*(X\pp X,\Delta(X);R)$$
is surjective.
\item The \emph{weak topological complexity}, $\wTC(X)$, is the least integer $n$, such that the composition 
$(1\pp q_n)(1\pp\Delta_n)$ is fibrewise pointed nullhomotopic.
\item The \emph{$\TC$-conilpotency of $X$}, $\cTC(X)$, is the least integer $n$, such that the composition 
$(1\pp\Sigma q_n)(1\pp\Sigma\Delta_n)$ is fibrewise pointed nullhomotopic.
\item For any integer $j\geq 0$ let $\sigma^j\TC(X)$ be the least integer $n$ for which there exists a fibrewise map 
$s\colon X\pp\Sigma^jX\rightarrow X\pp\Sigma^jG_n(X)$, such that the composition $(1\pp\Sigma^jp_n)s$ is fibrewise 
pointed homotopic to the identity. In particular, $\sigma^0\TC(X)=\TC(X)$. The \emph{$\sigma$-topological complexity of 
$X$} (or \emph{stable topological complexity of $X$}) is $$\sigma\TC(X)=\inf_{j\in\mathbb{N}}\sigma^j\TC(X).$$
\item The \emph{weak topological complexity of Ganea}, $\wTC_G(X)$, is the least integer $n$, such that $1\pp q_n'$ is 
fibrewise pointed nullhomotopic.
\end{enumerate}
\end{definition}

Table \ref{compare} relates these new lower bounds to the corresponding notions for the LS-category.
\begin{table}[ht]
\begin{center}\begin{tabular}{lr}
{LS-cat} & {$\TC^M$}\\\hline\\
$e_R(X)$ (Toomer invariant) & $e^{\TC}_R(X)$\\
$\wcat(X)$ (weak category) & $\wTC(X)$\\
$\conil(X)$ (conilpotency) & $\cTC(X)$\\
$\sigma\cat(X)$ ($\sigma$-category) & $\sigma\TC(X)$\\
$\wcat_G(X)$ (weak category of Ganea) & $\wTC_G(X)$\\[-2mm]
 & \\\hline\\
\end{tabular}\end{center}
\caption{Lower bounds for LS-category and their $\TC^M$ counterparts.}
\label{compare}
\end{table}

\begin{remark}
A fibrewise pointed map $f\colon E_1\rightarrow E_2$ between the fibrewise pointed spaces 
$B\stackrel{s_i}{\longrightarrow} E_i\stackrel{p_i}{\longrightarrow}B$ is \emph{fibrewise pointed nullhomotopic} 
if it is fibrewise pointed homotopic to $s_2\circ p_1$. To avoid an excessive use of indexing we write simply 
$f\simeq *$ and just keep in mind that we are working in the category of the fibrewise pointed spaces.
\end{remark}

Let $X\stackrel{s}{\rightarrow} X\pp Y\stackrel{p}{\rightarrow} X$ be any of the fibrewise pointed constructions 
described in Section \ref{sec2}. We will denote by $X\pp \Sigma Y$ the fibrewise pointed suspension of $X\pp Y$. 
To compare the lower bounds we just introduced we need the following result:

\begin{prop}\label{MVlemma}
For the homology and cohomology with any coefficients we have the following isomorphisms
$$H_n(X\pp Y,s(X))\iso H_{n+j}(X\pp\Sigma^j Y,(\Sigma^j s)(X))$$
and $$H^n(X\pp Y,s(X))\iso H^{n+j}(X\pp\Sigma^j  Y,(\Sigma^j s)(X)).$$
Moreover, if $f\colon X\pp Y\rightarrow X\pp Z$ is a fibrewise pointed map then there is a commutative diagram
$$\xymatrix{
H_n(X\pp Y,s(X)) \ar[rr]^{H_n(f)} \ar[d]_\cong & & H_n(X\pp Z,s(X)) \ar[d]_\cong\\
{H_{n+j}(X\pp\Sigma^j Y,(\Sigma^j s)(X))} \ar[rr]^{H_{n+j}(\Sigma^jf)}& & H_{n+j}(X\pp\Sigma^j Z,(\Sigma^j s)(X))
}$$
and a similar result holds for cohomology.
\end{prop}

\proof
Recall that $(X\pp\Sigma Y,(\Sigma s)(X))$ is obtained by gluing of the fibrewise pointed cones 
$(X\pp C^{+}Y,(C^{+}s)(X))$ and $(X\pp C^{-}Y,(C^{-}s)(X))$. The two cones are fibrewise pointed contractible, 
so $H_n(X\pp C^{+}Y,(C^{+}s)(X))\iso H_n(X\pp C^{-}Y,(C^{-}s)(X))=0$. The intersection of the two cones is 
$(X\pp Y,s(X))$, so from the relative version of the Mayer-Vietoris sequence for homology we can conclude that
$$H_n(X\pp Y,s(X))\iso H_{n+1}(X\pp\Sigma Y,(\Sigma s)(X)).$$
The analogous argument for cohomology deals with the second statement.
\endproof

As a preparation for the next theorem we list the relations between the lower bounds that are clear from the 
definition.

\begin{prop}\label{basic}
\begin{enumerate}
\item $\TC^M(X)\geq e^{\TC}_R(X)$,
\item $\TC^M(X)\geq\wTC(X)\geq\cTC(X)$,
\item $\TC^M(X)\geq\wTC_G(X)\geq\sigma^i\TC(X)\geq\sigma^j\TC(X)\geq\sigma\TC(X)$ for $1\leq i\leq j$,
\item $\wTC_G(X)\geq\wTC(X)$.
\end{enumerate}
\end{prop}

\proof
\begin{enumerate}
\item If $1\pp p_n\colon X\pp G_n(X)\rightarrow X\pp X$ has a section then the maps
$$H_*(1\pp p_n)\colon H_*(X\pp G_n(X),\widehat{G}_n(X);R)\rightarrow H_*(X\pp X,\Delta(X);R)$$
are surjective, 
so $\TC^M(X)\geq e^{\TC}_R(X)$.
\item If there exists  a map $s\colon X\pp X\rightarrow X\pp W^nX$ such that $(1\pp i_n)s\simeq 1\pp\Delta_n$, 
then
$$(1\pp q_n)(1\pp\Delta_n)\simeq(1\pp q_n)(1\pp i_n)s\simeq *,$$
so $\TC^M(X)\geq\wTC(X)$. If 
$(1\pp q_n)(1\pp\Delta_n)\simeq *$, then $(1\pp\Sigma q_n)(1\pp\Sigma\Delta_n)\simeq *$, so $\wTC(X)\geq\cTC(X)$.
\item Observe that $\wTC_G(X)=\sigma^1\TC(X)$. Indeed, it follows from the fibrewise Barratt-Puppe sequence 
\cite[\S I.2.12]{Crabb-James}
$$X\pp G_n(X)\stackrel{1\pp p_n}{\longrightarrow} X\pp X\stackrel{1\pp q_n'}{\longrightarrow} X\pp G_{[n]}(X)\stackrel{\delta}{\longrightarrow}$$
$$\rightarrow X\pp\Sigma G_n(X)\stackrel{1\pp \Sigma p_n}{\longrightarrow} X\pp\Sigma X\stackrel{1\pp\Sigma q_n'}{\longrightarrow} X\pp\Sigma G_{[n]}(X)\longrightarrow\ldots$$
that the map $1\pp q_n'$ is fibrewise pointed nullhomotopic if and only if the map $1\pp\Sigma p_n$ admits a homotopy 
section. Let $1\leq i\leq j$. If the map
$$1\pp p_n\colon X\pp G_n(X)\rightarrow X\pp X$$
admits a section, then there 
exists a map
$$s\colon X\pp X\rightarrow X\pp G_n(X),$$
such that $(1\pp p_n)s\simeq 1$, and so 
$(1\pp\Sigma^ip_n)(\Sigma^is)\simeq 1$. So, $\TC(X)\geq\sigma^i\TC(X)$.

If $(1\pp\Sigma^ip_n)(\Sigma^is)\simeq 1$, 
then $(1\pp\Sigma^jp_n)(\Sigma^js)\simeq 1$, so $\sigma^i\TC(X)\geq\sigma^j\TC(X)$. It is now obvious that 
$$\TC^M(X)\geq \wTC_G(X)\geq\sigma^i\TC(X)\geq\sigma^j\TC(X).$$
The rightmost inequality follows immediately from the definition of $\sigma\TC(X)$.
\item If $(1\pp q_n')\simeq *$, then
$$(1\pp q_n)(1\pp\Delta_n)\simeq(1\pp\tilde{\Delta}_n)(1\pp q_n')\simeq *,$$
so $\wTC_G(X)\geq\wTC(X)$.
\end{enumerate}
\endproof

The following theorem summarizes all known relations between the various lower bounds considered in this paper. 

\begin{theorem}\label{comp}
Let $R$ be a ring. Then
$$\nil_R(X)\leq\sigma\TC(X)\leq\wTC_G(X)\leq\TC^M(X)\;{\rm and}$$
$$\nil_R(X)\leq\cTC(X)\leq\wTC(X)\leq\wTC_G(X)\leq\TC^M(X).$$
If $k$ is a field, then
$$\nil_k(X)\leq e_k^{\TC}(X)\leq\sigma\TC(X).$$
\end{theorem}

\proof
If we compare the inequalities above with those from Proposition \ref{basic}, we see that we only have to show the 
following four:
\begin{center}\begin{tabular}{ll}
(a) $\nil_R(X)\leq\sigma\TC(X)$,& (b) $\nil_R(X)\leq\cTC(X)$,\\
(c) $\nil_k(X)\leq e_k^{\TC}(X)$,&(d) $e_k^{\TC}(X)\leq\sigma\TC(X)$.
\end{tabular}\end{center}
\begin{enumerate}
\item[(a)] Recall that $\sigma\TC(X)=\inf_{j\in\mathbb{N}}\sigma^j\TC(X)$. It therefore suffices to show that 
$\nil_R(X)\leq\sigma^j\TC(X)$ for all $j\in\mathbb{N}$ and this will imply (a). Fix $j$. 
If $\sigma^j\TC(X)=\infty$, then the statement obviously holds. So, let $\sigma^j\TC(X)=n<\infty$. Then there 
exists a map
$$s\colon X\pp\Sigma^jX\rightarrow X\pp\Sigma^jG_n(X),$$
such that $(1\pp\Sigma^jp_n)s\simeq 1$. We can 
conclude that the induced maps
$$H^*(1\pp\Sigma^jp_n)\colon H^*(X\pp\Sigma^jX)\rightarrow H^*(X\pp\Sigma^jG_n(X))$$
are injective. By Proposition \ref{MVlemma} the maps 
$H^*(1\pp p_n)$
are injective, so the maps $H^*(1\pp q'_n)$ are trivial. We conclude that 
$H^*(1\pp q_n)H^*(1\pp\Delta_n)=0$. Observe that the composition $(1\pp\Delta_n)(1\pp q_n)$ induces on cohomology 
precisely the $n$-fold cup-product on elements of positive dimension. So, $\nil_R(X)\leq n$.
\item[(b)] Let $\cTC(X)=n$. Then $(1\pp\Sigma q_n)(1\pp\Sigma\Delta_n)\simeq *$, so 
$$0=H^*((1\pp\Sigma q_n)(1\pp\Sigma\Delta_n))=H^*((1\pp q_n)(1\pp\Delta_n)).$$
Again by definition $\nil_R(X)\leq n$.
\item[(c)] Assume that $e^{\TC}_k(X)=n$ and the maps $H_*(1\pp p_n)$ are surjective. Then $H_*(1\pp q'_n)=0$ and 
$k$ is a field, so $H^*(1\pp q'_n)=0$. We conclude that $H^*(1\pp q_n)H^*(1\pp\Delta_n)=0$ and $\nil_k(X)\leq n$.
\item[(d)] The idea in this case is the same as in (a). Let $\sigma^j\TC(X)=n<\infty$. Then there exists a map 
$$s\colon X\pp\Sigma^jX\rightarrow X\pp\Sigma^jG_n(X),$$
such that $(1\pp\Sigma^jp_n)s\simeq 1$. Hence, 
$H_*(1\pp\Sigma^jp_n)$ are surjective. By Proposition \ref{MVlemma} the maps $H_*(1\pp p_n)$ are surjective and we 
have $e_k^{\TC}(X)\leq n$.
\end{enumerate}
\endproof


\section{Appendix}

In Section \ref{sec2} we used a common notation for a number of different constructions of fibrewise pointed spaces. 
In all cases the fibres were given by some functorial construction applied to a fixed space but with a varying choice 
of basepoints. The underlying intuition is that the base space acts on the fibres by 'moving around' the basepoints 
and that this action gives rise to some kind of a semi-direct product of the base with the fibre.

We are going to put this intuition on a firm basis by describing a general construction 
which starts with a continuous endofunctor $\Phi$ on the category of pointed topological spaces, and yields as a result 
a fibrewise pointed space $X\pp\Phi(X)$, whose fibres are the spaces $\Phi(X,x)$ for a varying choice of the basepoint $x$. 
To achieve the desired result we are going to assume that $X$ admits a closed embedding into some Euclidean space (or intrinsically, 
that $X$ is separable metric, locally compact and finite-dimensional). Such spaces are sufficiently general for  
the applications that we have in mind. 

Recall that an endofunctor $\Phi$ on the category of pointed topological spaces is \emph{continuous} if 
for each continuous function $f\colon Z\times X\rightarrow Y$ the corresponding function
$$\hat{f}\colon Z\times\Phi(X)\rightarrow \Phi(Y),$$
$$\hat f (y,u):=[\Phi(f(z,-))](u)$$
is continuous (cf. \cite[Chapter 2]{JamesGen}, see also \cite{IS2} for a slightly 
different version of continuity).

Let us assume that $X$ admits a closed embedding $i\colon X\rightarrow\mathbb{R}^m$ for some positive integer 
$m\in\mathbb{N}$ and define a family of pointed maps $\{i_x\}_{x\in X}$ by
$$\xymatrixcolsep{1.3cm}\xymatrix{
i_x\colon (X,x)\ar[r]^-{i} & (\mathbb{R}^m,i(x))\ar[r]^-{\tau_{-i(x)}} &(\mathbb{R}^m,0),
}$$
where $\tau_{-i(x)}\colon\mathbb{R}^m\rightarrow\mathbb{R}^m$ is the translation map that sends $i(x)$ to the origin $0$.
We then define $X\pp\Phi(X)$ to be the set 
$$\coprod_{x\in X}\Phi(X,x)$$ 
endowed with the initial topology  induced by the function
$$\xymatrixcolsep{1.5cm}\xymatrix{
\coprod_{x\in X}\Phi(X,x)\ar[r]^-{\coprod\Phi(i_x)} & X\times\Phi(\mathbb{R}^m,0).
}$$
(where $X\times\Phi(\mathbb{R}^m,0)$ has the product topology). The projection $p\colon X\pp\Phi(X)\to X$ is defined by 
sending each $\Phi(X,x)$ to $x$, while the section $s\colon X\to X\pp\Phi(X)$ is obtained by mapping each $x\in X$ to the basepoint of $\Phi(X,x)$.
We obtain the following commutative diagram
$$\xymatrixcolsep{2cm}\xymatrix{
X\pp\Phi(X)=  \coprod_{x\in X}\Phi(X,x)\ar[r]^-{\coprod\Phi(i_x)}\ar@<-1ex>[d]_p & X\times\Phi(\mathbb{R}^m,0)\ar@<-1ex>[d]_{\rm pr}\\
X \ar@{=}[r] \ar@<-1ex>[u]_s & X \ar@<-1ex>[u]_{s_0}
}$$
It is not difficult to check that both $p$ and $s$ are continuous. In fact, for an open set $U\subseteq X$ we have 
$p^{-1}(U)=(\coprod\Phi(i_x))^{-1}(U\times \Phi(\mathbb{R}^m))$ which is 
open by the definition of the topology on $X\pp\Phi(X)$. Similarly, every open set in $X\pp\Phi(X)$ is of the form 
$(\coprod\Phi(i_x))^{-1}(V)$ for some open set $V\subseteq X\times\Phi(\mathbb{R}^m,0)$, so 
$s^{-1}((\coprod\Phi(i_x))^{-1}(V))=s_0^{-1}(V)$ is open in $X$.

Observe that the fibre over $x\in X$ of the fibrewise space $X\pp\Phi(X)$ is the set $\Phi(X,x)$ endowed with the initial topology determined by the map 
$$\Phi(i_x)\colon \Phi(X,x)\to \Phi(\mathbb{R}^m,0),$$
which may in general be weaker than the topology on $\Phi(X,x)$. We will say that the endofunctor
$\Phi$ \emph{preserves Euclidean subspaces} if for every subspace $i\colon X\hookrightarrow \mathbb{R}^m$, and every $x\in X$ the space $\Phi(X,x)$ has the initial
topology with respect to the map $\Phi(i)\colon \Phi(X,x)\to\Phi(\mathbb{R}^m,x)$. 

\begin{remark}
In most applications (like those in Section \ref{sec2}) the functor $\Phi$ preserves injective maps, and then the above definition simply states that $\Phi$ 
preserves the subspace topology. One should also keep in mind that the property is often non-trivial, especially when the functor $\Phi$ involves quotient 
topologies, which are not well-behaved with respect to subspace topologies.
\end{remark}
The proof of the following proposition is now obvious. 

\begin{prop}
If $\Phi$ is an endofunctor of pointed topological spaces that preserves Euclidean subspaces, then $X\pp\Phi(X)$ is a fibrewise pointed space whose 
fibre over each $x\in X$ is the space $\Phi(X,x)$. 
\end{prop}

Although we must choose a specific embedding $i\colon X\to \mathbb{R}^m$ for the construction of $X\pp\Phi(X)$, the following proposition 
shows that the result does not depend on that choice.

\begin{prop}
Let $\Phi$ be a continuous endofunctor on the category of pointed topological spaces. If $\Phi$ preserves Euclidean subspaces, then the definition of the 
fibrewise pointed space $X\pp\Phi(X)$ is independent of the choice of a closed embedding $i\colon X\to\mathbb{R}^m$.
\end{prop}
\begin{proof}
Let us denote by $\iota\colon\mathbb{R}^m=\mathbb{R}^m\times 0\hookrightarrow \mathbb{R}^{m+n}$ the standard inclusion. Since $\Phi$ preserves Euclidean subspaces, 
the topology on $\Phi(\mathbb{R}^m,0)$ is precisely 
the initial topology with respect to the map $\Phi(\iota)\colon\Phi(\mathbb{R}^m,0)\to\Phi(\mathbb{R}^{m+n},0)$. Then clearly $X\times \Phi(\mathbb{R}^m,0)$ has the 
initial topology with respect to the map 
$$1_X \times \Phi(\iota)\colon X\times \Phi(\mathbb{R}^m,0)\longrightarrow X\times \Phi(\mathbb{R}^{m+n},0).$$
We may now compare the topologies on $\coprod\Phi(X,x)$ induced by the embeddings $i\colon X\to\mathbb{R}^m$ and $\iota\circ i\colon X\to \mathbb{R}^{m+n}$.
The preceding result together with the commutativity of the following diagram
$$\xymatrixcolsep{2cm}\xymatrix{
\coprod_{x\in X}\Phi(X,x)\ar[r]^-{\coprod\Phi(i_x)}\ar[dr]_{\coprod\Phi((\iota\circ i)_x)\ \ \ } & X\times\Phi(\mathbb{R}^m,0)\ar[d]^{1\times\Phi(\iota)}\\
 & X\times \Phi(\mathbb{R}^{m+n},0)
}$$
implies that the two topologies in fact coincide. We may therefore arbitrarily expand the codomain of the embedding. 

Let us now compare the embeddings $i\colon X\to \mathbb{R}^m$ and $j\colon X\to \mathbb{R}^n$, or equivalently (in view of the above considerations) the embeddings
$i\colon X\to \mathbb{R}^m\times 0\subset\mathbb{R}^{m+n}$ and $j\colon X\to 0\times \mathbb{R}^m\subset\mathbb{R}^{m+n}$. It is well-known that any two such 
embeddings are equivalent by an ambient homeomorphism, which can be defined as follows. We first use the Tietze extension theorem to extend the map $j\circ i^{-1}\colon i(X)\to\mathbb{R}^{m+n}$ 
to a map $\varphi\colon \mathbb{R}^{m+n}\to\mathbb{R}^{m+n}$, and to extend the map $i\circ j^{-1}\colon j(X)\to\mathbb{R}^{m+n}$ to a map 
$\psi\colon \mathbb{R}^{m+n}\to\mathbb{R}^{m+n}$. Then one can easily check that the map $h\colon\mathbb{R}^m\times\mathbb{R}^n\to\mathbb{R}^m\times\mathbb{R}^n$, given 
by $h(u,v):=(u-\psi(v+\varphi(u)),v+\varphi(u))$ is a homeomorphism and that the following diagram commutes
$$\xymatrixcolsep{1.5cm}\xymatrixrowsep{1.5cm}\xymatrix{
X \ar[r]^-j \ar[d]_i & \mathbb{R}^{m+n}\\
\mathbb{R}^{m+n} \ar[ur]_h
}$$
For every $x\in X$ we define a homeomorphism $h_x\colon\mathbb{R}^{m+n}\to\mathbb{R}^{m+n}$ by
$$h_x(u):=h(u+i(x))-j(x),$$
so that $h_x\circ i_x=j_x$,
and we obtain the commutative diagram 
$$\xymatrixcolsep{2cm}\xymatrixrowsep{1.5cm}\xymatrix{
\coprod_{x\in X}\Phi(X,x)\ar[r]^-{\coprod\Phi(j_x)}\ar[d]_{\coprod\Phi(i_x)} & X\times\Phi(\mathbb{R}^{m+n},0)\\
X\times \Phi(\mathbb{R}^{m+n},0)\ar[ur]_{\bar{h}}
}$$
where $\bar h$ is given by $\bar h(x,u):=(x,\Phi(h_x)(u))$; it is a homeomorphism by the continuity of the functor $\Phi$.
We conclude that the topologies on $\coprod_{x\in X}\Phi(X,x)$ induced by the two embeddings coincide.
\end{proof}

For a subspace $A\subseteq X$ we will denote by $A\pp\Phi(X)$ the restriction of the fibrewise pointed space $X\pp\Phi(X)$ over $A$. Then we have the following important property:

\begin{prop}
Let $\Phi$ be as above and let $A\subseteq X$. If $A\pp X$ is trivial as a fibrewise pointed space then so is $A\pp\Phi(X)$. 
\end{prop}
\begin{proof}
Triviality of $A\pp X$ means that there is a homeomorphism $f\colon A\pp X\to A\times (X,x_0)$ for which the following diagram commutes:
$$\xymatrixcolsep{2cm}\xymatrix{
A\pp X=  \coprod_{x\in A} (X,x)\ar[r]^-f\ar@<-1ex>[d]_{\rm pr} & A\times (X,x_0)\ar@<-1ex>[d]_{\rm pr}\\
A \ar@{=}[r] \ar@<-1ex>[u]_\Delta & X \ar@<-1ex>[u]_{s_0}
}$$
Clearly, $f$ must be of the form $f(a,x)=(a,f_a(x))$ so we can use $\Phi$ to construct another commutative diagram:
$$\xymatrixcolsep{2cm}\xymatrix{
A\pp \Phi(X)=  \coprod_{x\in A} \Phi(X,x)\ar[r]^-{\coprod\Phi(f_a)} \ar@<-1ex>[d]_p & A\times \Phi(X,x_0)\ar@<-1ex>[d]_{\rm pr}\\
A \ar@{=}[r] \ar@<-1ex>[u]_s & X \ar@<-1ex>[u]_{s_0}
}$$
The continuity of $\Phi$ implies that $\coprod\Phi(f_a)$ is a homeomorphism, which proves that $A\pp\Phi(X)$ is fibrewise pointed trivial.
\end{proof}

Now we know that for a suitable choice of $\Phi$ and $X$ the construction of the fibrewise pointed space $X\pp\Phi(X)$ is well-defined. It remains to
show that constructions introduced in Section \ref{sec2} are special instances of it.

\begin{example} (The product and the fat wedge)
Fix an embedding $i\colon X\to\mathbb{R}^m$ and let $\Phi=\Pi^n$ be the $n$-fold product functor. 
Clearly $\Pi^n$ is continuous and preserves Euclidean subspaces so we may define $X\pp\Pi^nX$ as above. On the other hand in Section \ref{sec2} we defined 
the fibrewise pointed space $X\pp\Pi^nX$ as the product space $X\times X^n$ with suitable projection and section maps. To compare the two constructions observe 
that the space $X\times X^n$, as a set, coincides with $\coprod_{x\in X}\Pi^n(X,x)$, and that we have the following diagram
$$\xymatrixcolsep{1.5cm}\xymatrix{
\coprod_{x\in X}\Pi^n(X,x)\ar[r]^-{\coprod \Pi^n(i_x)}\ar@{=}[d] & X\times ((\mathbb{R}^m)^n,0^n)\\
X\times X^n \ar[ur]_j
}$$
where the map $j$ is given by 
$$j(x,y_1,\ldots,y_n)=(x,i(y_1)-i(x),\ldots,i(y_n)-i(x)).$$ 
To prove that the two definitions coincide it is sufficient to show that $j$ is an embedding. This is clear, since both $j$ and its inverse
$j^{-1}\colon j(X\times X^n)\to X\times X^n$, which is given by 
$$j^{-1}(x,u_1,\ldots,u_n)=(x,i^{-1}(u_1+i(x)),\ldots,i(u_n+i(x)))$$ 
are clearly continuous.

The argument for the fibrewise pointed fat-wedge is analogous. The functor $W^n$ is continuous and preserves Euclidean subspaces, so we may define 
$X\pp W^nX$ using the general construction, while in Section \ref{sec2} we defined $X\pp W^nX$ as the subspace 
$\{(x,y_1,\ldots,y_n)\in X\times X^n\mid \exists k\ y_k=x\}$ of $X\times X^n$. As for the $n$-fold products, we directly verify that $j\colon X\pp W^nX\to X\times W^n(\mathbb{R}^m,0)$
is an embedding, therefore the two definitions of $X\pp W^nX$ coincide (see Fig. \ref{sop}).
\end{example}

\begin{figure}[ht]\begin{center}\begin{pspicture}(0.2,-1.1)(3.9,4.8)
\psset{linewidth=0.5pt,dash=2pt 2pt,unit=1cm,labelsep=3pt} 
\definecolor{siva}{gray}{1}\definecolor{svsiva}{gray}{0.7}\definecolor{tesiva}{gray}{0}
\uput[r](2.1,-0.2){$\pr_1$}
\uput[r](3.5,-1.1){$I$}
\uput[r](3,0.3){$I\pp W^2 I$}
\pstThreeDCoor[xMax=3,yMax=3,zMax=3,drawing=false]
\psset{origin={1,2},Alpha=65,Beta=15}
\pstThreeDBox[linecolor=svsiva](0,0,0)(3,0,0)(0,3,0)(0,0,3)
\pstThreeDLine[arrows=->](1.5,1.5,-1)(1.5,1.5,-2)
\pstThreeDLine(1.5,0,-2.5)(1.5,3,-2.5)
\psset{linecolor=svsiva,linestyle=dashed}
\pstThreeDSquare(0,0,0)(3,0,0)(0,0,3)
\pstThreeDSquare(0,1,0)(3,0,0)(0,0,3)
\pstThreeDSquare(0,2,0)(3,0,0)(0,0,3)
\pstThreeDSquare(0,3,0)(3,0,0)(0,0,3)
\pstThreeDLine[linecolor=tesiva,linestyle=solid](1.5,0,1.5)(1.5,3,1.5)
\psset{linecolor=tesiva,linewidth=1.2pt,linestyle=solid}
\pstThreeDLine(1.5,0,1.5)(1.5,0,3)
\pstThreeDLine(0,0,1.5)(1.5,0,1.5)
\pstThreeDLine(2,1,1.5)(0.5,1,1.5)
\pstThreeDLine(1.5,1,2.5)(1.5,1,1)
\pstThreeDLine(2.5,2,1.5)(1,2,1.5)
\pstThreeDLine(1.5,2,0.5)(1.5,2,2)
\pstThreeDLine(1.5,3,1.5)(1.5,3,0)
\pstThreeDLine(1.5,3,1.5)(3,3,1.5)
\psset{linewidth=0.5pt}
\pstThreeDLine(1.5,0,3)(1.5,3,1.5)
\pstThreeDLine(1.5,0,1.5)(1.5,3,0)
\pstThreeDLine(0,0,1.5)(1.5,3,1.5)
\pstThreeDLine(1.5,0,1.5)(3,3,1.5)
\end{pspicture}\end{center}
\caption{\small An easy example of a fibrewise pointed (fat) wedge $X\pp W^nX$ (for $X=I$, $n=2$) that is not locally trivial as a fibrewise space.
The space $I\pp W^2 I$ has the subspace topology with respect to $\mathbb{R}\pp \Pi^2 \mathbb{R}=\mathbb{R}^3$. It consists of two 
rectangles which intersect along one of the diagonals. The fibre over each $x\in I$ is the wedge $W^2(I,x)$ obtained from 
two intervals by identifying the basepoints $x$. Note that the fibres over $0$ and $1$ are not homeomorphic to the fibres over the interior points of the interval.}
\label{sop}\end{figure}
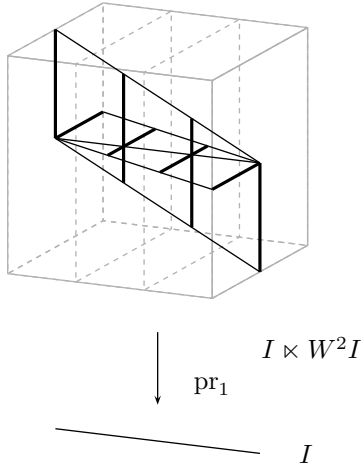

\begin{example} (The smash product) The $n$-fold smash product functor is continuous, and it preserves Euclidean subspaces by \cite[Theorem 6.2.1]{Dug}
(since $X$ is embedded in $\mathbb{R}^m$ as a closed subspace). In Section 2 we defined $X\pp\wedge^n X$ as the fibrewise quotient of the space $X\pp\Pi^nX$ 
over the subspace $X\pp W^n(X)$. Using the notation from the previous example we can form the following diagram
$$\xymatrixcolsep{1.6cm}\xymatrix{
X\pp W^mX \ar@{_(->}[d]_{1\pp\iota} \ar[r]^{j|} & X\times W^n(\mathbb{R}^m,0)\ar@{_(->}[d]^{1\times\iota}\\
X\pp \Pi^mX \ar[d]_{1\pp q} \ar[r]^j & X\times \Pi^n(\mathbb{R}^m,0) \ar[d]^{1\times q}\\
X\pp \wedge^mX \ar@{-->}[r]^{\bar j} & X\times \wedge^n(\mathbb{R}^m,0)
}$$
Observe that $1\pp q$ and $1\times q$ are both quotient maps, the first by definition, and the second because $X$ is locally compact. 
As above we conclude that the induced map $\bar j$ is an embedding, and that the two definitions of $X\pp\wedge^nX$ agree.
\end{example}

\begin{example}(The path space) We have formerly defined the fibrewise path space over $X$ as the path space $X^I$ together with 
the projection $X^I\to X$ given by the evaluation $\ev_0$, and the section $X\to X^I$, that to every $x\in X$ assigns the constant path at $x$.
As a set, we may identify $X^I$ with $\coprod_{x\in X}P(X,x)$, so in order to prove that $X^I=X\pp PX$ we must show that 
the map $j\colon X^I\to X\times P(\mathbb{R}^m,0)$, given by $j(\alpha):=\bigl(\alpha(0),i\circ\alpha-i\circ c_{\alpha(0)}\bigr)$ is an embedding. 
The map $j$ is clearly continuous, and so is its inverse $j^{-1}\colon j(X^I)\to X^I$, since it is given by $j^{-1}(x,\beta):=i^{-1}\circ (\beta+c_{i(x)})$.
\end{example}


\begin{thebibliography}{99}
\bibitem{CLOT}
O.~Cornea, G.~Lupton, J.~Oprea, D.~Tanr\'e, \emph{Lusternik-Schnirelmann Category}, \rm{AMS, Mathematical Surveys and Monographs, vol. {\bf 103} (2003).}
\bibitem{Crabb-James}
M.~Crabb, I.~James, \emph{Fibrewise homotopy theory}, (Springer Verlag, London, 1998).
\bibitem{Dold:book}
A.~Dold, \emph{Lectures on Algebraic Topology}, (Springer Verlag, Berlin, Heidelberg, 1995).
\bibitem{Dold}
A.~Dold, \emph{Partitions of unity in the theory of fibrations}, \rm{Ann. of Math. {\bf 78} (1963), 223--255.}
\bibitem{Dug}
J.~Dugundji, \emph{Topology}, (Boston: Allyn and Bacon, 1966).
\bibitem{Farber:TC}
M.~Farber, \emph{Topological complexity of motion planning}, \rm{Discrete Comput. Geom. {\bf 29} (2003), 211--221.}
\bibitem{Farber:Robotics}
M.~Farber, \emph{Invitation to Topological Robotics}, (EMS Publishing House, Zurich, 2008)
\bibitem{GV}
J.~M.~Garc\'ia Calcines, L.~Vandembroucq, \emph{Weak sectional category}, \rm{Journal of the London Math. Soc. {\bf 82}(3) (2010), 621--642.}
\bibitem{IS}
N.~Iwase, M.~Sakai, \emph{Topological complexity is a fibrewise LS category}, \rm{Topology Appl. {\bf 157}(2010), 10-21.}
\bibitem{ISerrata}
N.~Iwase, M.~Sakai, \emph{Erratum to ``Topological complexity is a fibrewise LS category'' [Topology Appl. 157, No 1, 10--21 (2010)]}, \rm{Topology Appl. {\bf 159}(2012), 2810-2813.}
\bibitem{IS2}
N.~Iwase, M.~Sakai, \emph{Functors on the category of quasi-fibrations}, \rm{Topology Appl. {\bf 155}(2008), 1403-1409.}
\bibitem{JamesGen}
I.M.~James, \emph{General topology and homotopy theory}, \rm{Springer Verlag, 1984.}
\bibitem{James}
I.M.~James, \emph{Fibrewise topology}, \rm{Cambridge tracts in mathematics {\bf 91}, 1989.}
\bibitem{JamesMorris}
I.M.~James, J.R.~Morris, \emph{Fibrewise category}, \rm{Proc. Roy. Soc. Edinburgh, {\bf 119A} (1991), 177--190.}
\bibitem{Schwarz}
A.S.~Schwarz, \emph{The genus of a fiber space}, \rm{Amer. Math. Soc. Transl. (2) {\bf 55} (1966), 49--140.}
\end{thebibliography}
\end{document}